\documentclass{amsart}
\usepackage[utf8]{inputenc}
\usepackage{amsmath,amssymb,amsthm}
\usepackage{thm-restate}
\usepackage{url}
\usepackage{cleveref}
\usepackage[all,cmtip]{xy}
\usepackage{graphicx,color,fullpage}

\usepackage[dvipsnames]{xcolor}

\newcommand{\R}{\mathbb{R}}
\newcommand{\N}{\mathbb{N}}
\newcommand{\C}{\mathbb{C}}
\newcommand{\hyp}{\mathcal{H}}

\newtheorem{thm}{Theorem}[section]
\newtheorem{lemma}[thm]{Lemma}
\newtheorem{conj}[thm]{Conjecture}

\theoremstyle{definition}

\newtheorem{cor}[thm]{Corollary}
\theoremstyle{definition} 
\newtheorem{definition}[thm]{Definition}
\newtheorem{rem}[thm]{Remark}
\numberwithin{equation}{section}
\title{Symmetric Hyperbolic Polynomials}
\author{Grigoriy Blekherman}
\author{Julia Lindberg}
\author{Kevin Shu}

\DeclareMathOperator{\Sym}{Sym}

\DeclareMathOperator{\disc}{Disc}
\DeclareMathOperator{\argmax}{argmax}

\thanks{Grigoriy Blekherman and Kevin Shu were partially supported by NSF grant DMS-1901950.}

\begin{document}
\maketitle
\begin{abstract}
Hyperbolic polynomials have been of recent interest due to applications in a wide variety  of fields. We seek to better understand these polynomials in the case when they are symmetric, i.e. invariant under all permutations of variables. We give a complete characterization of the set of symmetric hyperbolic polynomials of degree 3, and a large class of symmetric hyperbolic polynomials of degree 4. For a class of polynomials, which we call hook-shaped, we relate symmetric hyperbolic polynomials to a class of linear maps of univariate polynomials preserving hyperbolicity, and give evidence toward a beautiful characterization of all such hook-shaped symmetric hyperbolic polynomials.
We show that hyperbolicity cones of a class of symmetric hyperbolic polynomials, including all symmetric hyperbolic cubics, are spectrahedral. Finally, we connect testing hyperbolicity of a symmetric polynomial to the degree principle for symmetric nonnegative polynomials.
\end{abstract}
\section{Introduction}
A homogeneous polynomial $p \in \R[x_1, \dots, x_n]$ is \emph{hyperbolic} with respect to some $v \in \R^n$ if $p(v) \neq 0$, and for any $x \in \R^n$, the univariate polynomial $p_x(t) = p(x + tv)$ has only real roots.
Hyperbolic polynomials are surprisingly ubiquitous in modern mathematics. They played important roles in the proofs of the Kadison-Singer conjecture \cite{marcus2015interlacing}, the Schrijver-Valiant conjecture \cite{gurvits2006hyperbolic}, and many recent results in sampling theory \cite{anari2016monte}. Moreover, hyperbolic polynomials (and the closely related stable polynomials) have inspired many fruitful generalizations, such as the notion of Lorentzian polynomials which were defined concurrently in \cite{branden2020lorentzian} and \cite{anari2018log}.

Despite their ubiquity, understanding the structure of the set of hyperbolic polynomials is difficult. It is known that hyperbolic polynomials form a contractible set in the vector space of all homogeneous polynomials in a given degree and number variables \cite{nuij1968a}.
On the other hand, it was shown in \cite{saunderson2019certifying} that it is NP-hard to determine if a given polynomial is hyperbolic with respect to a given direction $v$ (even if the degree of the polynomial is of degree 3).

Here, we consider polynomials $p$ that are \emph{symmetric}, i.e, invariant with respect to all permutations of variables, and also hyperbolic with respect to the all 1's vector $\vec{1}$.
We refer to such polynomials as \emph{symmetric hyperbolic} polynomials.

Symmetric hyperbolic polynomials are of particular interest due to \cite{bauschke2001hyperbolic}, which showed that symmetric hyperbolic polynomials can be used to construct many examples of hyperbolic polynomials.
\begin{thm}[Theorem 3.1 of \cite{bauschke2001hyperbolic}]
    Let $q(x)$ be a polynomial of degree $d$ which is hyperbolic with respect to $v \in \R^n$. Let $\lambda_1(x), \dots, \lambda_d(x) \in \R$ be the roots of the polynomial $q_x(t) = q(x + tv)$. If $p \in \R[y_1, \dots, y_d]$ is symmetric hyperbolic, then $p(\lambda_1(x), \dots, \lambda_d(x))$
    is hyperbolic with respect to $v$.
\end{thm}

Some of our results concern the \emph{hyperbolicity cone} of a hyperbolic polynomial. The hyperbolicity cone of a polynomial $p$ hyperbolic with respect to $v$ is the set
\[
    H_v(p) = \{x \in \R^n : p(x+tv) \neq 0 \text{ for all }t>0\}.
\]
It was shown in \cite{gaarding1959inequality} that the hyperbolicity cone of a hyperbolic polynomial is always convex. The Generalized Lax conjecture states that the hyperbolicity cone of every hyperbolic polynomial is spectrahedral, i.e. that  for some $m \in \N$, there exist matrices $A_1, \dots, A_n \in \Sym(\R^m)$ so that
\[
    H_v(p) = \{x \in \R^n : \sum_{i=1}^n A_i x_i \succeq 0\}.
\]
The Lax conjecture stated that every hyperbolic polynomial in 3 variables has a \emph{definite determinantal representation}, meaning that the polynomial is of the form $p(x) = \det(\sum_{i=1}^n A_i x_i)$, for some symmetric matrices satisfying $\sum_{i=1}^n A_i v_i \succ 0$.
The Lax conjecture was proved by Helton and Vinnikov in \cite{helton2007linear}, and this implies the Generalized Lax Conjecture for 3 variable polynomials.
There are hyperbolic polynomials, such as the elementary symmetric polynomials, which have no such definite determinantal representation.

One reason for the interest in the hyperbolicity cones of polynomials is that they are linear slices of the cone of nonnegative polynomials, and they can be used to certify that other polynomials are nonnegative as shown in \cite{kummer2015hyperbolic}.
Concretely, if $u,w \in H_v(p)$, then the \emph{mixed derivative} $$\Delta_{uw}p = D_u p D_wp - pD_{uw}p$$  is globally nonnegative, where $D_up$ denotes the directional derivative of $p$ in the $u$ direction. A polynomial $p$ is said to be \emph{weakly SOS-hyperbolic} (a phrase coined in \cite{saunderson2019certifying}) if $\Delta_{uw} p$ is a sum of squares for all $u,w \in H_v(p)$. If any power of $p$ has a determinantal representation, then $p$ is SOS-hyperbolic \cite[Proposition 4.7]{saunderson2019certifying}.

There has been much work concerning the elementary symmetric polynomials, defined as\\ $e_k = \sum_{\atop{S \subseteq [n],}{|S| = k}} \prod_{i \in S} x_i$.
For example, it was shown in \cite{zinchenko2008hyperbolicity} that the hyperbolicity cone of of $e_{n-1}$ is a spectrahedral shadow, which was improved in \cite{sanyal2013derivative} to show that the hyperbolicity cone of $e_{n-1}$ is in fact spectrahedral. These results were generalized in \cite{branden2014hyperbolicity} where it was shown that the hyperbolicity cone of $e_k$ is spectrahedral for any $1 \le k \le n$, and this spectrahedral representation was simplified and extended to other polyomials in \cite{kummer2021spectral}.

Our contributions to the study of symmetric hyperbolic polynomials include a connection between symmetric hyperbolic polynomials and the theory of univariate hyperbolicity preservers originally due to Polya and Schur in \cite{schur1914zwei}. This connection allows us to neatly characterize symmetric hyperbolic polynomials of degree three, as well as symmetric hyperbolic polynomials of degree four whose expansion into elementary symmetric polynomials only involves hook-shaped partitions. We call such symmetric polynomials \emph{hook-shaped polynomials}. In degree five, we exhibit interesting hook-shaped hyperbolic polynomials which are counterexamples to some generalizations of our results in degrees 3 and 4. We offer a compelling conjecture for a characterization of hook-shaped hyperbolic polynomials of arbitrary degree, and extensive computational evidence supporting the conjecture.

In addition, we show that a class of symmetric hyperbolic polynomials, including all symmetric cubic hyperbolic polynomials, have spectrahedral hyperbolicity cones. We also show a version of the degree principle, as stated in \cite{timofte2003positivity} that applies to hyperbolic polynomials.

\subsection{Main Results in Detail}
While we will give some general results about symmetric hyperbolic polynomials, including the connection to degree principle, in \Cref{sec:general} our work focuses on \emph{hook-shaped polynomials} which form a linear subspace of symmetric hyperbolic polynomials.
\begin{definition}  
A homogeneous symmetric polynomial of degree $d$ is \emph{hook-shaped} if it is of the form
\[
    p(x) = \sum_{i=1}^d a_i e_1^{d-i}(x) e_{i}(x),
\]
where $e_i$ is the elementary symmetric polynomial of degree $i$, and $a_i \in \R$ for $i = 1,\dots,d$. We denote by $\Gamma_{n,d}$ the vector space of hook-shaped symmetric polynomials in $n$ variables of degree $d$.
\end{definition}  

The term `hook-shaped' originates from the study of partitions, as these basis polynomials correspond to hook-shaped partitions of $d$. Note that all cubic symmetric polynomials are hook-shaped.

Hook-shaped polynomials can be associated to certain linear maps between vector spaces of univariate polynomials.
We outline this connection after defining a subspace of the space of univariate polynomials, which we can think of as being those polynomials of degree $n$ whose roots sum to 0.
\begin{definition}  
\label{def:zero_sum_polys}
Let $\R[t]_{n,0}$ denote the vector space of univariate polynomials of degree at most $n$ with the coefficient of $t^{n-1}$ equal to 0. Let $\hyp_{n,0}$ denote the set of polynomials in $\R[t]_{n,0}$ with only real roots.
\end{definition}  
Note that $\R[t]_{n,0}$ contains all polynomials of degree $ \le n-2$.

\begin{definition}  
A linear map $T: \R[t]_{n,0} \rightarrow \R[t]_{d,0}$ is a \emph{0-sum hyperbolicity preserver} if $T(\hyp_{n,0}) \subseteq \hyp_{d,0}$.
The map $T$ is called diagonal if there exist $\gamma_1, \dots, \gamma_d \in \R$ so that $T(t^{n-k}) = \gamma_k t^{d-k}$ for all $k = 0\dots d$, and $T(t^{n-k}) = 0$ for $k > d$. 
\end{definition}  
\begin{definition}  
Let $p \in \Gamma_{n,d}$. The \emph{associated operator} to $p$ is the function $T : \R[t]_{n,0} \rightarrow \R[t]_{d,0}$ defined as follows.
If $g(t) \in \R[t]_{n,0}$ is a polynomial so that 
\[g(t) = a(t-r_1)(t-r_2)\dots(t-r_n),\]
for $r_1, \dots, r_n \in \C$, $a \in \R \neq 0$ then we let  $T(g)(t) = ap(\vec{r} - \vec{1}t)$,
where $\vec{r}$ denotes the vector whose entries are the roots $r_i$. We then extend this definition to all $g\in \R[t]_{n,0}$ (including those of degree less than $n$) by continuity.
\end{definition}  
Although it is not immediately clear from the definition, we establish that for any $p \in \Gamma_{n,d}$ the associated operator $T$ is a diagonal linear map, and $T$ is a $0$-sum hyperbolicity preserver if and only if $p$ is symmetric hyperbolic.
\begin{restatable}{thm}{hookop}
\label{thm:hookop}
Let $p \in \Gamma_{n,d}$, with the associated operator $T$. Then $T$ is a diagonal linear map.
Moreover, $p$ is symmetric hyperbolic if and only if $T$ is a 0-sum hyperbolicity preserver. The map sending a hook-shaped polynomial to its associated operator is linear and invertible. 
\end{restatable} 

In light of this theorem, we study symmetric hyperbolic polynomials in terms of diagonal 0-sum hyperbolicity preservers.
We compare the 0-sum hyperbolicity preservers to the set of \emph{hyperbolicity preservers}, which are simply maps $T :\R[t]_n \rightarrow \R[t]_d$ that preserve real rootedness for all univariate polynomials of degree at most $n$.
The diagonal hyperbolicity preservers were already characterized in the work of Schur and Polya in \cite{schur1914zwei}. This work was extended in various ways in subsequent work, culminating in the work of Borcea and Brand\"en in \cite{borcea2009lee, borcea2009polya}, which characterized not only (nondiagonal) linear hyperbolicity preservers of univariate polynomials, but also linear hyperbolicity preservers of multivariate multiaffine polynomials as well.

\begin{definition} 
Let $T : \R[t]_{n,0} \rightarrow  \R[t]_{d,0}$ be a diagonal, 0-sum hyperbolicity preserver. We say $T$ is \emph{extendable} if there exists a diagonal linear map $\hat{T} : \R[t]_{n} \rightarrow  \R[t]_{d}$ so that $\hat{T}$ is a hyperbolicity preserver, and $T$ is the restriction of $\hat{T}$ to $\R[t]_{n,0}$.
\end{definition} 

The following result shows that diagonal, $0$-sum hyperbolicity preservers are always extendable to full hyperbolicity preservers only if the degree $d$ is at most $4$, and there exist non-extendable $0$-sum hyperbolicity preservers if the degree is at least $5$.

\begin{thm} 
\label{thm:extendable}
Let $T:\R[t]_{n,0} \rightarrow  \R[t]_{d,0}$ be a 0-sum hyperbolicity preserver. If $d \le 4$, then $T$ is extendable, and moreover, this is the case if and only if $T((x-1)^{n-1}(x+n-1))$
has real roots, where $d-1$ of them have the same sign.

If $d \ge 5$, then there exists a 0-sum hyperbolicity preserver which is not extendable.
\end{thm} 

Using \Cref{thm:extendable} we can immediately obtain characterizations of hook-shaped symmetric hyperbolic polynomials of degree at most $4$. We first state the result for cubics:
\begin{cor}
Let $p$ be a cubic, symmetric polynomial. $p$ is symmetric hyperbolic if and only if $p(u + \vec{1}t)$ is real rooted, where $u$ is any coordinate vector.
\end{cor} 
Note that this enables us to check that a cubic symmetric polynomial is hyperbolic by checking if a single univariate polynomial is hyperbolic. 
As an immediate consequence we can give a semi-algebraic description of all symmetric hyperbolic cubics.
\begin{cor}
    Let $p \in \R[x_1,\ldots,x_n]_3$ be a hyperbolic cubic polynomial of the form
    \[ p = a \Tilde{e_1}^3 + b \Tilde{e_1} \Tilde{e_2}+ c \Tilde{e_3} \]
    where $\Tilde{e_k}(x) = \binom{n}{k}^{-1} e_k^n$. Then $p$ is hyperbolic if and only if $(a+b+c) \left(27 a c^2-b^3-9 b^2 c\right) \le 0$.
\end{cor}
\begin{proof}
   Let $u = [1,0,\ldots,0]^T$. By direct computation we see:
    \[
        \disc(p(u + t \vec{1})) =  
        -4(a+b+c) \left(27 a c^2-b^3-9 b^2 c\right).
    \]
    Since a cubic polynomial is real rooted if and only if its discriminant is nonnegative, we obtain the result.
\end{proof}

\Cref{thm:extendable} also immediately characterizes hook-shaped quartic polynomials.
\begin{cor} 
\label{thm:quartics}
Let $p \in \Gamma_{n,4}$. 
Let $q(t) = p(u + \vec{1}t)$, where $u$ is a coordinate vector.
$p$ is symmetric hyperbolic if and only if $q(t-\frac{1}{n})$ is real rooted, with at least $3$ roots having the same sign.
\end{cor} 
Finally, we say more about a quintic example of a non-extendable $0$-sum hyperbolicity preserver and its associated hool-shaped hyperbolic polynomial.
\begin{thm} 
\label{thm:quintic_example}
    Let 
    \[
        p = 4500 e_5 - 220 e_1 e_4 + 7 e_1^2 e_3,
    \]
    which is in $5$ variables.
    Then $p$ is symmetric hyperbolic, $p$'s associated operator is not extendable, $\Delta_{\vec{1}\vec{1}}p$ is SOS, and $p$ is not SOS hyperbolic.
\end{thm}

Finally,  
we show that a class of hook-shaped polynomials have spectrahedral hyperbolicity cones.
\begin{restatable}{thm}{spectrahedral}
\label{thm:spectrahedral}
    If $\ell(x)$ is any linear functional such that $\ell(\vec{1}) \ge 0$, then the polynomial $e_k(x) + \ell(x)e_{k-1}(x)$ is hyperbolic with respect to $\vec{1}$, and its hyperbolicity cone is spectrahedral.
\end{restatable}
This result suffices to show that all cubic symmetric hyperbolic polynomials have spectrahedral hyperbolicity cones.
\begin{cor}
    If $p$ is a symmetric hyperbolic cubic polynomial, then $p$ has a spectrahedral hyperbolicity cone.
\end{cor}
\section{Conjectures and Open Problems}
Our results leave open several interesting questions. First, we might wonder whether or not the natural extension of the Polya-Schur characterization to 0-sum hyperbolicity preservers holds:
\begin{conj}
    \label{conj:main}
    Let $T : \R[t]_{n,0} \rightarrow \R[t]_{d,0}$ be a diagonal linear map. Then $T$ is a 0-sum hyperbolicity preserver if and only if $T((x-1)^{n-1}(x+n-1))$
    has real roots with $d-1$ having the same sign.
\end{conj}
If true, this conjecture would give a characterization of all hook-shaped symmetric hyperbolic polynomials. 

We show that \Cref{conj:main} holds for $d \le 4$, and also that the sign condition on the roots of $T((x-1)^{n-1}(x+n-1))$ is necessary for all $d$.
Furthermore, we have extensive computational evidence that this conjecture holds when $d \le 6$, using the following procedure. We chose a real rooted polynomial $q\in \R[t]_{d,0}$ with $d-1$ roots of the same sign. For this polynomial $q$, there is a unique diagonal map $T$ satisfying $T((x-1)^{n-1}(x+n-1)) = q$, and a unique hook-shaped symmetric polynomial $p$ whose associated operator is $T$. We then verify that $\Delta_{\vec{1}, \vec{1}} p$ is a sum of squares, which (together with some additional properties of $p$) implies that $p$ is symmetric hyperbolic.
This leads us to an additional conjecture.
\begin{conj}
    If $p \in \Gamma_{n,d}$, then $p$ is symmetric hyperbolic if and only if $\Delta_{\vec{1}, \vec{1}} p $ is SOS.
\end{conj}

A more speculative conjecture is as follows:
\begin{conj}
    If $p \in \Gamma_{n,d}$ is symmetric hyperbolic, then the associated operator of $p$ is extendable if and only if $p$ is weakly SOS-hyperbolic.
\end{conj}
We have a large amount of computational evidence that for cubics and quartics this holds in the sense that we cannot find any examples of such polynomials which are not SOS-hyperbolic. Our evidence in the case of quintic polynomials is more limited, but we were not able to find a counterexample.

\section{General Results about Symmetric Hyperbolic Polynomials}
\label{sec:general}
While most of our results concern the specific class of hook-shaped symmetric polynomials, we first present some general results about symmetric hyperbolic polynomials. These results are useful for our investigation of hook-shaped polynomials, but they may be of general interest as well.

First, we justify our choice of defining a symmetric hyperbolic polynomial to be one which is hyperbolic with respect to $\vec{1}$ (as opposed to a general vector).
\begin{lemma}
\label{lem:irred_sym}
    If $f$ is an irreducible symmetric hyperbolic polynomial, and $f$ is hyperbolic with respect to some $v \in \R^n$, then $f$ is hyperbolic with respect to $\vec{1}$.
\end{lemma}
\begin{proof}
   By \cite{kummer2019connectivity}, if $f$ is irreducible, the only hyperbolicity cones of $f$ are $H_v(f)$ and $-H_v(f)$.

    Let $S$ be the stabiliser of $H_v(f)$ in $S_n$, which we see must be index at most 2, because every element of $S_n$ must send $H_v(f)$ to itself or to $-H_v(f)$.
    In particular, $S$ must act transitively on $[n]$.

    Since $v$ is in the interior of $H_v(f)$, there must be some $v' \in H_v(f)$ so that $e_1(v') \neq 0$. We then have that $\vec{1} =  \frac{n}{|S|e_1(v')} \sum_{\sigma \in S}\sigma v'$ is an element of $H_v(f)$ by convexity, and so $f$ is hyperbolic with respect to $\vec{1}$.
\end{proof}

We define the \emph{elementary symmetric mean} of degree $k$ in $n$ variables to be $$\tilde{e}^n_k = \frac{1}{\binom{n}{k}}e^n_k.$$
We will often suppress the dependence of this expression on $n$ and simply write $\tilde{e}_k$ when $n$ can be safely left implicit. This notation will make many expressions in terms of elementary symmetric polynomials independent of the number of variables.
For example, we have the binomial expansion type expression for any $n$:
\begin{equation}
\label{eq:norm_elem_expansion}
    \tilde{e}_k(x + t \vec{1}) = \sum_{i=0}^k \binom{k}{i} \tilde{e}_i(x) t^{k-i}.
\end{equation}

Any symmetric poynomial can be expressed as a polynomial in elementary symmetric means. We now show that if a polynomial is symmetric hyperbolic, then the same expression in elementary symmetric means is also symmeetric hyperbolic for any larger number of variables.

\begin{lemma}
\label{lem:n_bigger}
    Let $f$ be a symmetric hyperbolic polynomial in $n$ variables. Express $f$ as a polynomial in elementary symmetric means, so  $f = q(\tilde{e}^n_1, \tilde{e}^n_2, \dots, \tilde{e}^n_n)$, for some polynomial $q$.
    For all $m \ge n$, the polynomial $f_m = q(\tilde{e}^m_1, \tilde{e}^m_2, \dots, \tilde{e}^m_n)$ is also symmetric hyperbolic.
\end{lemma}
\begin{proof}
Note that we may expand
\[
f_m(x + t\vec{1}) = \sum_{i = 0}^d q_k(\tilde{e}^m_1(x), \tilde{e}^m_2(x), \dots, \tilde{e}^m_n(x))t^{k-i},
\]
where $q_k$ is a polynomial whose coefficients crucially do not depend on $m$, due to \Cref{eq:norm_elem_expansion}.

In \cite{rosset1989normalized} it is shown that for any $x \in \R^m$, there exists $y \in \R^n$ such that for every $i \in [n]$, $\tilde{e}_i^n(y) = \tilde{e}_i^m(x)$. Therefore there exists some $y \in \R^n$ so that
\[
f_m(x + t\vec{1}) = f(y + t\vec{1}).
\]
This is then real rooted for every $x$ since $f$ is symmetric hyperbolic.
\end{proof}

We now prove an analogue of the degree principle for globally nonnegative polynomials (originally proved in \cite{timofte2003positivity}) for hyperbolic polynomials.

\begin{thm}
    Let $p$ be a homogeneous symmetric polynomial of degree $d$ with $p(\vec{1}) \neq 0$. Then $p$ is symmetric hyperbolic if and only if for every $x$ with at most $d-1$ distinct entries, the univariate polynomial $p(x+t\vec{1})$ has only real roots.
\end{thm}
\begin{proof}
    We need to show that if $p(x+t\vec{1})$ is not real rooted for some $x$, then there exists $x'$ with at most $d-1$ distinct entries so that $p(x'+t\vec{1})$ is not real rooted.

    We use induction on the degree $d$: if $d = 1$, then the result is vacuous in the sense that all nonzero symmetric linear polynomials are hyperbolic with respect to $\vec{1}$. Now, we may assume that the theorem holds for polynomials of degree $d-1$.
    In particular, if $D_{\vec{1}}p$ is not symmetric hyperbolic, then by the inductive hypothesis, there exists $v$ with at most $d-2$ distinct entries such that $D_{\vec{1}}p(v+t\vec{1})$ is not real-rooted, and therefore, $p(v+t\vec{1})$  is also not real-rooted. So, we may assume that $D_{\vec{1}}p$ is symmetric hyperbolic.

    Assume now that $p(x+t\vec{1})$ is not real rooted for some $x$.
    Since $p(x+t\vec{1})$ is not real-rooted while $D_{\vec{1}}p(x+t\vec{1})$ is real rooted, there exists either a local minimum of $p(x+t\vec{1})$ which is strictly positive, or a local maximum of $p(x+t\vec{1})$ which is strictly negative. By replacing $p$ by its negative if necessary, we may assume that there is a local minimum of $p(x+t\vec{1})$ which is strictly positive.
    
    Since $p$ is symmetric in $x$, each coefficient in $t$ of $p(x+t\vec{1})$ is symmetric, so there are polynomials $f_0, \dots, f_d$ such that 
    \[
        p(x+t\vec{1}) = \sum_{i=0}^d f_i(e_1(x),  \dots, e_i(x))t^{d-i}.
    \]
    By degree considerations, each $f_i$ is a polynomial which is linear in $e_i$. The only appearance of $e_d(x)$ in this expression is in the $t^0$ coefficient. Together, these two observations imply that there is a polynomial $q$ in $d$ variables so that 
    \[
        p(x+t\vec{1}) = c e_d(x) + q(e_1(x), e_2(x), \dots, e_{d-1}(x), t).
    \]
    
    Suppose that $x' \in \R^n$ has the property that $e_i(x) = e_i(x')$ for $i = 1, \dots, d-1$, then
    \begin{align*}
        p(x'+t\vec{1}) &= c e_d(x') + q(e_1(x'), e_2(x'), \dots, e_{d-1}(x'), t)\\
        &= c (e_d(x') - e_d(x)) + ce_d(x) + q(e_1(x), e_2(x), \dots, e_{d-1}(x), t)\\
        &=c(e_d(x')-e_d(x)) + p(x+t\vec{1}).
    \end{align*}

    Now, we note that by \cite{RIENER2012850}, for any $c \in \R$, there is an optimal solution to the following optimization problem with at most $d-1$ distinct entries:
    \begin{align}
        \argmax_{x'} &\; \{c e_d(x'):\;e_i(x') = e_i(x)\text{ for }i \in [d-1]\}.
    \end{align}

    Now, if $x'$ optimizes this, then $p(x'+t\vec{1})$ is not real rooted, since $p(x'+t\vec{1}) =  p(x'+t\vec{1}) + d$, for some $d > 0$, so that there is a local minimum of $p(x'+t\vec{1})$ which is strictly positive.
\end{proof}

\section{Hook-shaped Symmetric Hyperbolic Polynomials and Univariate Hyperbolicity Preservers}
\subsection{Background on Hyperbolicity Preservers}
We denote by $\R[t]_n$ the $n+1$ dimensional vector space of univariate polynomials of degree at most $n$.
A hyperbolicity preserver is a linear map $T:\R[t]_n \rightarrow \R[t]_d$ such that for every real rooted polynomial $g \in \R[t]_n$, $T(g)$ is real rooted.
A linear map $T:\R[t]_n \rightarrow \R[t]_d$ is diagonal if there exist $\gamma_0,\gamma_1, \dots, \gamma_d \in \R$ so that $T(x^{n-i})  = \gamma_i x^{d-i}$ for $i \le d$, and $T(x^{n-i}) = 0$ for $i > d$.

The following was shown in  \cite{schur1914zwei}:
\begin{thm}
\label{thm:schurpolya}
    Let $T : \R[t]_n \rightarrow \R[t]_d$ be a diagonal linear map.
    Then $T$ is a hyperbolicity preserver if and only if $T((x-1)^n)$ has real roots, all with the same sign.
\end{thm}

This result was greatly generalized in \cite{borcea2009polya} and \cite{borcea2009lee}, which considered not only real rooted univariate polynomials but also the more general multivariate stable polynomials.
For our purposes, we will only require the case of this theorem concerning diagonal maps between univariate polynomials.

\subsection{Associated Operators and 0-Sum Hyperbolicity Preservers}
In this subsection, we show \Cref{thm:hookop}, restated here.
\hookop*

\begin{proof}
    Recall the definition of the elementary symmetric means: $\tilde{e}_i(x) = \frac{1}{\binom{n}{i}} e_i(x)$.
    
    Fix $p \in \Gamma_{n,d}$, so that $p = \sum_{i=1}^d a_i\tilde{e}_1(x)^{d-i}\tilde{e}_i(x)$. Let $g(t) \in \R[t]_{n,0}$ be monic with roots $r_1, \dots, r_n \in \C$.
    We may write $g(t) = \prod_{i=1}^{n} (t - r_i) = \sum_{i=0}^n \binom{n}{k} c_i t^{n-i}$.

    Observe that that $c_i = \tilde{e}_i(r_1, \dots, r_n)$.
    Since $g \in \R[t]_{n,0}$, we have that $c_1 = 0$.
    Now, consider the associated operator 
    \[
        T(g) = p(\vec{r} - \vec{1}t) = \sum_{i=1}^d a_i\tilde{e}_1(\vec{r} - \vec{1}t)^{d-i}\tilde{e}_i(\vec{r} - \vec{1}t).
    \]
    It follows from Taylor expanding $\tilde{e}_i(\vec{r} - \vec{1}t)$ in $t$ that $\tilde{e}_i(\vec{r} - \vec{1}t) = \sum_{j=0}^i (-1)^{i-j}\binom{i}{j}\tilde{e}_j(\vec{r})t^{i-j}$, and in particular, $\tilde{e}_1(\vec{r} - \vec{1}t) = -t$.
    We then compute that
    \begin{align*}
        T(g) &= p(\vec{r} - \vec{1}t)\\
        &= \sum_{i=1}^d a_i\left(\sum_{j=0}^i(-1)^{d-j}\binom{i}{j}\tilde{e}_j(\vec{r})t^{d-j}\right)\\
        &= 
        \sum_{j=0}^d \left(\sum_{i=j}^d(-1)^{d-j} \binom{i}{j}a_i\right)\tilde{e}_j(\vec{r})t^{d-j}\\
        &= 
        \sum_{j=0}^d \left(\sum_{i=j}^d(-1)^{d-j} \binom{i}{j}a_i\right)c_jt^{d-j}
    \end{align*}
    That is, $T(g) = \sum_{i=0}^d \gamma_j c_j t^{d-j}$, where $ \gamma_j = \frac{1}{\binom{n}{j}}\sum_{i=j}^d(-1)^{d-j}\binom{i}{j} a_i$.
    
    After extending by linearity to all elements of $\R[t]_{n,0}$, $T$ is a diagonal linear map.
    It is not hard to see that the linear map that sends $(a_1, \dots, a_d)$ to $(\gamma_1, \dots, \gamma_d)$ is upper triangular with nonzero diagonal entries, and therefore invertible. We see then that the map $A$ is linear and bijective.

    We now show that $T$ is a 0-sum hyperbolicity preserver if and only if $p$ is symmetric hyperbolic.
    
    If $p$ is symmetric hyperbolic, and $g = \prod_{i=1}^n (t-r_i)$ is monic, then $p(\vec{r} - \vec{1}t)$ has only real roots, which is exactly saying that $T(g)$ has only real roots. If $g \in \R[t]_{n,0}$ is not monic, then we obtain that $T(g)$ is real rooted by homogeneity and continuity of $T$.

    On the other hand, let $T(g)$ be real rooted for all $g \in \R[t]_{n,0}$ with real roots.
    If $x \in \R^n$ with $e_1(x) = 0$, we find that $T(\prod_{i=1}^n (t-x_i)) = p(x - t \vec{1})$ has only real roots.
    To show that $p$ is hyperbolic, fix $x \in \R^n$, and consider the univariate polynomial $p(x - t \vec{1})$. Letting $\hat{x} = x - \tilde{e}_1(x) \vec{1}$, with $e_1(\tilde{x}) = 0$,  we obtain $p(x - t \vec{1}) = p(\tilde{x} - (t-\tilde{e}_1(x)) \vec{1})$, which is real rooted since $p(\tilde{x} - t \vec{1})$ is.
\end{proof}
\subsection{A Quintic Example}
Here, we exhibit a quintic symmetric polynomial which will be of interest to us.
\begin{lemma}
    Let $p$ be the symmetric polynomial in $n \ge 5$ variables defined by
    \[
        p = 6 \tilde{e}_5 - \frac{22}{3} \tilde{e}_1 \tilde{e}_4 + \frac{7}{3} \tilde{e}_1^2 \tilde{e}_3.
    \]
    $p$ is symmetric hyperbolic. If $n= 5$, then $p$ is not SOS-hyperbolic.
\end{lemma}
\begin{proof}
Our proof of both of these facts about $p$ will rely heavily on explicit computations.
    By \Cref{lem:n_bigger}, it suffices to show that $p$ is hyperbolic when $n = 5$.

    In this case, we note that $D_{\vec{1}}p$ is irreducible, and in particular, it is square free. It can be verified using \Cref{thm:quartics}  that $D_{\vec{1}}p$ is symmetric hyperbolic.

    It can be seen by reduction to the univariate case that if $p$ is a homogeneous polynomial so that $D_{\vec{1}}p$ is hyperbolic and square-free, and $\Delta_{\vec{1}\vec{1}}p$ is globally nonnegative, then $p$ is hyperbolic. 
    By using the \texttt{SumsOfSquares.jl} \cite{legat2017sos} package we checked directly that $\Delta_{\vec{1}\vec{1}}p$ is a sum of squares, which implies that $p$ is hyperbolic. Also using the package we verified that for $n = 5$, $\Delta_{\vec{1}w}p$ is not a sum of squares, where $w = (6,1,1,1,1)$.
\end{proof}
\begin{rem}
    This Lemma, together with \Cref{lem:nonextendable} shows Theorem \ref{thm:quintic_example}.
\end{rem}
\section{0-Sum Hyperbolicity Preservers}
\subsection{Sign Conditions and $g_0(t)$}
In this section, we use properties of the univariate polynomial
\[
g_0(t) = (x+n-1)(x-1)^{n-1} = \sum_{k=0}^n -(k-1)\binom{n}{k}t^{n-k} \in \R[t]_{n,0}.
\]
We first show is a necessary condition for a diagonal map to be a 0-sum hyperbolicity preserver.
\begin{lemma}
    \label{lem:necessary}
    Let $T : \R[t]_{n,0}\rightarrow \R[t]_{d,0}$ be a 0-sum hyperbolicity preserver. Then $T(g_0(t))$ has at least $d-1$ roots which have the same sign.
\end{lemma}
\begin{proof}
    Suppose that $T(x^{n-i}) = \gamma_i x^{d-i}$ for each $i$.

    Note that the only positive coefficient of $g_0(t)$ is that of $t^n$.
    Because $T(g_0(t))$ is real rooted, it suffices by Descartes' rule of signs that the $\gamma_i$ either all have the same sign or they alternate in sign.

    We first claim that $\gamma_0$ has the same sign as $\gamma_2$. To see this, note that 
    \[
        T((t-1)(t+1)t^{n-2}) = \gamma_0 t^d - \gamma_2 t^{d-2} = t^{d-2}(\gamma_0 t^2-\gamma_2).
    \]
    This has real roots if and only if $\gamma_0$ and $\gamma_2$ have the same sign.

    Now, consider the diagonal linear map $T' : \R[t]_{n-2} \rightarrow \R[t]_{d-2}$ given by $T'(g) = T(g)$ for $g \in \R[t]_{n-2}$.
    We have that $T'$ is a hyperbolicity preserver, so  \Cref{thm:schurpolya} implies that the coefficients of $T'((x+1)^{n-2})$ have real roots of the same sign, which implies by Descarte's rule of signs that $\gamma_{i}$ has the same sign as $\gamma_{i+2}$ for each $i \ge 2$.
\end{proof}

\subsection{An Equivalent Condition of Extendability for 0-Sum Hyperbolicity Preservers}
Let $T : \R[t]_{n,0} \rightarrow \R[t]_{d,0}$ be a 0-sum hyperbolicity preserver.
Recall that $T$ is extendable if there exists $\hat{T} : \R[t]_n \rightarrow \R[t]_d$ such that $\hat{T}$ preserves hyperbolicity and $\hat{T}(g) = T(g)$ for all $g \in \R[t]_{n,0}$.
We will give an equivalent characterization of extendability in this section and then later use it to show \Cref{thm:extendable}.

Note that if $T : \R[t]_{n,0} \rightarrow  \R[t]_{d,0}$ is diagonal, then it is uniquely determined by the image of $g_0(t)$ because all of the coefficients of $g_0(t)$ besides that of $t^{n-1}$ are nonzero.
If $g \in \R[t]_{d,0}$, we will let $T_g$ denote the unique diagonal map so that $T_{g}(g_0(t)) = g(t)$. 

We now define a family of maps which play an important role in our investigation of hyperbolicity preservers.

\begin{definition}
Let $\delta_n :\R[t]_n \rightarrow \R[t]_{n,0}$ be the diagonal linear map defined by $$\delta_n(t^{n-k}) = -(k-1)t^{n-k}$$ for all $k \in [n].$  
\end{definition}
Observe that $\delta_n((t-1)^n) = g_0(t)$. Moreover, for any diagonal map $T : \R[t]_{n} \rightarrow \R[t]_{d}$, $T(\delta_n(g)) = \delta_d(T(g))$,
since the coefficient of $\delta_n(t^{n-k})$ does not depend on $n$.

\begin{lemma}
\label{lem:roots_extendable}
    Let $g \in \R[t]_{d,0}$.
    The map $T_g$ is extendable if and only if there exists $f \in \R[t]_d$ such that $\delta_d(f) = g$ and $f$ has real roots, all of which have the same sign.
\end{lemma}
\begin{proof}
    First suppose that $T_g$ is extendable, so there exists a diagonal hyperbolicity preserver $\hat{T}$ with $\hat{T}(g) = T(g)$ for all $g \in \R[t]_{n,0}$.

    We use the properties of $\delta_d$ to show that
    \begin{align*}
    g(t) &=  T_g(g_0(t))\\
        &=  \hat{T}(g_0(t))\\
        &=  \hat{T}(\delta_n((x-1)^n))\\
        &=  \delta_d(\hat{T}((x-1)^n))
    \end{align*}
    Let $f = \hat{T}((x-1)^n)$. By \Cref{thm:schurpolya}, $f$ has real roots with the same sign, and $g = \delta_d(f)$, as desired.

    Next suppose that there exists an $f$ with real roots of the same sign and such $\delta_n(f) = g$.

    We let $\hat{T}$ be unique diagonal linear map sending $(x-1)^n$ to $f$. 
    By \Cref{thm:schurpolya}, $\hat{T}$ is then a hyperbolicity preserver, and moreover
    \[
        \hat{T}(g_0) = 
        \hat{T}(\delta_n((x-1)^n)) = 
        \delta_d(\hat{T}((x-1)^n)) = 
        \delta_d(f) = g.
    \]
    Therefore, $\hat{T}$ must restrict to $T_g$ on $\R[t]_{n,0}$ giving our desired extension.
\end{proof}

We have now seen the importance of the map $\delta_n$ in understanding extendable linear maps. We will give some more properties of this map in the next part.

\subsection{Properties of the map $\delta_n$ and a nonextendable map}
Let $R_n : \R[x]_n \rightarrow \R[x]_n$ be the map that reverses the order of the coefficients, i.e. $R_n(p) = t^n p\left(\frac{1}{t}\right)$.
and let 
\[
D(p) = R_{n-1} \left(\frac{d}{dt}R_n(p)\right).
\]

We can also view the operator $D$ in the following way:  homogenize $p$ to a bivariate homogeneous polynomial $\bar{p}(t,s)$, then $D(p)$ is $\frac{d}{ds}\bar{p}$ evaluated at $s=1$. More explicitly:
$$D(p)=\left[\frac{d}{ds}\bar{p}\right]_{s=1}.$$

\begin{lemma}
\label{lem:factoring}
\[
    \delta_n(p) = p  -  D(p)
\]
\end{lemma}
\begin{proof}
    Note that $D(p)$ is a diagonal map, as
    \[
        D(t^k) = R_{n-1}\left(\frac{d}{dt}t^{n-k}\right) = 
        R_{n-1}\left((n-k)t^{n-k-1}\right) = 
       (n-k)t^{k}.
    \]
    We have that $p - D(p)$, when applied to $t^{n-k}$ gives
    \[
        t^{n-k} - D(t^{n-k}) = (1-k)t^{n-k} = \delta_n(t^{n-k}),
    \]
    which proves the claim by linearity.
\end{proof}
We will see that $\delta_n$ in fact has a number of interesting properties in relation to the multiplicities of roots.

We recall that if $g \in \hyp_n$ and $q \in \hyp_{n-1}$, then $g$ and $q$ \emph{interlace} if 
\[
    r_1 \le s_1 \le r_2 \le s_2 \dots \le s_{n-1} \le r_n,
\]
where $r_i$'s are the roots of $g$ and $s_i$'s are the roots of $q$. We say that $g$ and $q$ \emph{strictly interlace} if they interlace and all of the above inequalities are strict.

\begin{lemma}
\label{lem:interlace_delta}
    For all $p \in \hyp_n$ with nonnegative roots , the polynomials $D(p)$ and $p$ interlace, and also $D(p)$ and $\delta_n(p)$ interlace. 
\end{lemma}

\begin{proof}
   It suffices to prove the Lemma for all $p$ with distinct, positive real roots.
    The full theorem then follows by taking limits, since the set of polynomials with distinct positive real roots is dense in the set of polynomials with nonnegative real roots, and the property we want to conclude is closed.
    
    Under this assumption, $\frac{d}{dt}(R_np)$ strictly interlaces $R_np$ by Rolle's theorem.
    
    If $p$ is a polynomial of degree exactly $n$ with roots $0 < r_1 < r_2 < \dots < r_n$, then $R_np$ has roots 
    \[
        0 < r_n^{-1} < r_{n-1}^{-1} < \dots < r_1^{-1}.
    \]
    Similarly, we know that if $s_1 < s_2 < \dots < s_{n-1}$ are the roots of $D(p)$, then the roots of $\frac{d}{dt}(R_np) = R_{n-1}D(p)$ are 
    \[
        s_{n-1}^{-1} < 
        s_{n-2}^{-1} < \dots <
        s_{1}^{-1}
    \]

    Interlacing then implies that
    \[
        0 < r_n^{-1} < s_{n-1}^{-1} < r_{n-1}^{-1} < \dots <s_1^{-1}< r_1^{-1}.
    \]
    Inverting, we see that
    \[
        0 < r_1 < s_1 < r_2 < \dots <s_{n-1}< r_n.
    \]
    This implies that $p$ and $D(p)$ interlace.

    Next, we recall Obreschkoff's theorem \cite{obreschkoff1963verteilung} which states that a polynomial $q$ interlaces $p$ if and only if for all $\alpha, \beta \in \R$, we have that $\alpha q + \beta p$ is real rooted. So, to show $D(p)$ and $\delta_n(p)$ interlace, we only need to show that for all $\alpha, \beta$,  $\alpha(p - D(p)) + \beta D(p) =  \alpha p + (\beta - \alpha)D(p))$
    has real roots. This follows because $p$ and $D(p)$ interlace.
\end{proof}

\begin{lemma}
    \label{lem:mult1}
    Let $p \in \hyp_n$ have nonnegative real roots. Suppose that $p$ vanishes at  $r$ with multiplicity exactly $k$. If $r\neq 0$, then $\delta_n(p)$ has a root at $r$ with multiplicity exactly $k-1$. If $r = 0$, then $\delta_n(p)$ has a root at 0 with multiplicity at least $k$.
\end{lemma}
\begin{proof}
    In this proof, when we say a polynomial vanishes at a point with multiplicity $k$, we mean that it vanishes with multiplicity exactly equal to $k$.

    If $r > 0$ and $p$ vanishes at $r$ with multiplicity $k$, then $D(p)$ vanishes at $r$ with multiplicity $k-1$.
    Next, since $D(p)$ and $p$ both vanish at $r$ to multiplicity at least $k-1$ we see that $\delta_n(p) = p - D(p)$
    also vanishes to order at least $k-1$ at $r$.
    On the other hand, because $p$ vanishes to order $k$ at $r$, but $D(p)$ does not, $\delta_n(p)$ cannot vanish to order $k$ at $r$, so $\delta_n(p)$ vanishes at $r$ with multiplicity exactly $k-1$.

    If $r = 0$, then the first $k$ coefficients of $p$ are 0, and we see that the last $k$ coefficients of $R_np$ are 0. This implies that the first $k$ coefficients of $\frac{d}{dt} R_np$ are 0, and that the first $k$ coefficients of $D(p)$ are 0. Therefore, we conclude that both $p$ and $D(p)$ vanish with multiplicity $k$ at 0. We conclude that $\delta_n(p)$ also vanishes at 0 with multiplicity at least $k$. 
\end{proof}

Using the above Lemma, we exhibit a nonextendable zero-sum hyperbolicity preserver.
\begin{lemma}
\label{lem:nonextendable}
    Let $n \ge 5$.
    Let $T_g$ be the unique diagonal map sending $(t-(n-1))(t-1)^{n-1}$ to $(t-1)^2(t-2)^2(t+6)$, then $T$ is a zero-sum hyperbolicity preserver, but $T$ is not extendable.
\end{lemma}
\begin{proof}
    We start by showing that $T$ is not extendable.

    Suppose that $T$ is extendable. Then by \Cref{lem:roots_extendable}, there exists $f \in \R[t]_5$ with nonnegative real roots such that 
    \[
        \delta_5(f) = (t-1)^2(t-2)^2(t+6).
    \]
    By \Cref{lem:mult1}, we see that $f$ has a root of multiplicity 3 at 1 and also a root of multiplicity 3 at 2, but this is a contradiction, since $f$ has degree 5.

    Now, we wish to show that $T$ is a hyperbolicity preserver. 
    For this, we note that the symmetric polynomial $p = 6 \tilde{e}_5 - \frac{22}{3} \tilde{e}_1 \tilde{e}_4 + \frac{7}{3} \tilde{e}_1^2 \tilde{e}_3$ has $T$ as an associated operator, and since $p$ is hyperbolic by \Cref{thm:quintic_example}, $T$ is a hyperbolicity preserver, as desired.
\end{proof}
\subsection{A Topological Proof of Extendability}
We now give a proof that all 0-sum hyperbolicity preservers from $\R[t]_{n,0}$ to $\R[t]_{4,0}$ are extendable. While it may be possible to obtain explicit formulas for such an extension, we give a nonconstructive proof based on the observations we made in the last section.

Define the family of polytopes $A_d$ as follows: 
\[
    A_d = \{(r_1, \dots, r_d) : r_1 \ge r_2 \ge \dots \ge r_d \ge 0, \,\, \sum_{i=1}^d r_i = 1\}.
\]
This polytope can be thought of as a `projectivization' of the set of univariate real-rooted polynomials with all roots of the same sign.
Explicitly, let $\mathcal{H}_{+, d}$ be the set of all real rooted univariate polynomials $p$ of degree $d$ with roots $r_1 \ge r_2 \ge \dots \ge r_d \ge 0$ such that $p$ is not a scalar multiple of $t^d$. There is an action of the multiplicative group $\R^{\times} \times \R^{\times}_+$ on $\mathcal{H}_{+, d}$ where $(\alpha, \beta) \cdot p = \alpha p(\beta t)$. We then have that $A_d$ is homeomorphic to the quotient of $\mathcal{H}_{+, d}$ by this action, and indeed the map $\rho$ defined below  induces a homeomorphism from the quotient space to $A_d$.

We can also think of $A_{d-1}$ as being a projectivization of the set of zero-sum univariate polynomial with all but one root having the same sign, in an analogous way. If we let $\mathcal{H}_{+,0, d}$ denote the set of all zero-sum real rooted univariate polynomials with $d-1$ roots of the same sign, then we have an analogous action of  $\R^{\times} \times \R_+$ on $\mathcal{H}_{+,0,d}$, and $A_{d-1}$ is homeomorphic to the quotient space. As a consequence of \Cref{lem:interlace_delta} the map $\delta_d$ sends real rooted univariate polynomials of degree $d$ with roots of the same sign to zero-sum univariate polynomials with $d-1$ roots of the same sign. 

We would like for $\delta_d$ to also induce a map between the associated projective spaces of roots. However, $\delta_d(t^d - t^{d-1}) = t^d$, which does not correspond to a point in the associated projective space.

To rectify this, we define a map  $\phi : A_d \rightarrow A_{d-1}$
as follows: set $\phi(1,\underbrace{0,\dots,0}_{d-1 \text{ times}}) = (1,\underbrace{0,\dots,0}_{d-2 \text{ times}})$, and for all other vectors $r=(r_1,\dots,r_d)$, let
\[
    \phi(r_1, \dots, r_d) = \rho\left(\delta_d\left(\prod_{i=1}^d(t-r_i)\right)\right).
\]

where
\[
    \rho(p) = \frac{1}{\sum_{i=1}^{d-1}r_i}(r_1, \dots, r_{d-1}),
\]
 for a monic polynomial $p \in \hyp_{d,0}$ with roots $r_1, \dots, r_d$.

\begin{lemma}
    The map $\phi$ is well defined and continuous.
\end{lemma}
\begin{proof}
    We first show that $\phi$ is well defined, i.e. that $\phi(r) \in A_{d-1}$ for any $r \in A_d$.

    Fix $(r_1, \dots, r_d) \in A_d$, and let $p(t) = \prod_{i=1}^d (t-r_i)$.
    By \Cref{lem:interlace_delta}, 
    $D(p)$ interlaces $p$, so in particular, $D(p)$ has nonnegative real roots. 
    
    By \Cref{lem:interlace_delta}, $D(p)$ interlaces $\delta_d(p)$, which implies that $\delta_d(p)$ has $d-1$ nonnegative roots.
    If moreover, $p \neq t^n$, then $\delta_d(p)$ has a positive root, and so, $\rho(\delta_d(p)) \in A_{d-1}$ is well defined.

    Because the roots of a polynomial are continuous in the coefficients of that polynomial, on $A_d \setminus \{(1,0,\dots,0)\}$, $\phi$ is continuous. It remains to show that $\lim_{r \rightarrow (1,0,\dots,0)} \phi(r) = (1,0,\dots,0)$.

    Fix some $1 > \delta > 0$. We will show that there is some $\epsilon$ with the property that if $r \in A_d$ with $r \neq (1,0,\dots,0)$ and $1-r_1 < \epsilon$, then $\|\phi(r) - (1,0,\dots,0)\| < \delta$.
    Very roughly speaking, we will accomplish this by arguing that the largest entry of $\phi(r)$ is much larger than any other entry of $\phi(r)$ when $1-r_1$ is small. This implies that after normalizing $\phi(r)$ to have sum equal to 1, the largest entry will have value tending toward 1 as $r$ tends to $(1,0, \dots, 0)$

    More concretely, let $s_1 \ge s_2 \ge \dots \ge s_d$ be the roots of $\delta_d(p)$.
    We will show that $s_1 \ge C \sqrt{1-r_1}$ whereas $s_i < c (1-r_1)$ for $d > i > 1$, for some constants $C, c > 0$, whenever $1-r_1$ is small enough. It is clear that this then implies that as $1-r_1$ tends toward 0, $\phi(r)_1 = \frac{s_1}{\sum_{i=1}^{d-1} s_i}$ will tend toward 1, and $\phi(r)_i = \frac{s_1}{\sum_{i=1}^{d-1} s_i}$ will tend toward 0. This implies the result.
    
    Because $r_i \ge 0$ for all $i$ and $\sum_{i=1}^n r_i = 1$, $r_i < 1-r_1$ for each $i$.
    It can be seen from the proof of \Cref{lem:interlace_delta} that $s_i \le r_i$ for each $1 < i < d$.
    This implies that $s_i \le 1-r_1$ for all $i$.

    Now, we must show that $s_1 \ge C \sqrt{1-r_1}$ for some $C > 0$ whenever $1-r_1$ is small enough. To do this, we note from Newton's identity that
    \begin{align*}
        2e_2(r) &= \left( \sum_{i=1}^d r_i\right)^2 - \sum_{i=1}^d r_i^2\\
        &= 1 - \sum_{i=1}^d r_i^2\\
        &>  1 - r_1^2 - \sum_{i=2}^d (1-r_1)^2\\
        &= (1-r_1)(1+r_1-(d-1)(1-r_1))\\
        &> 1-r_1.
    \end{align*}
    This last inequality holds whenever $1-r_1$ is small enough.

    $p(t) = t^d - t^{d-1} + e_2(r)t^{d-2}+\dots$, so by definition $\delta_d(p) = t^d - e_2(r)t^{d-2}+\dots$.
    Because $s_1, \dots, s_d$ are the roots of $\delta_d(p)$, we also have that
    \[
        \delta_d(p) = t^d + e_2(s)t^{d-2}+\dots.
    \]
    We conclude $-e_2(s) = e_2(r)$. This implies that
    \[
        -2e_2(s) = \sum_{i=1}^d s_i^2 - \left( \sum_{i=1}^d s_i\right)^2 = \sum_{i=1}^d s_i^2 > 1-r_1.
    \]
    Using the fact that $0 \le s_i \le 1-r_1$ for each $d > i > 2$, we have that
    \[
        s_1^2 + s_d^2 + (d-2)(1-r_1)^2 > (1-r_1),
    \]
    or equivalently,
    \[
        s_1^2 + s_d^2 \ge (1-r_1)(1 - (d-2)(1-r_1)) > \frac{1}{2}(1-r_1).
    \]
    This inequality holds whenever $1-r_1$ is small enough.
    
    We also have that $\sum_{i=1}^d s_i = 0$ so that $s_1 + s_d = -\sum_{i=2}^{d-1} s_i \ge - (d-2)(1-r_1)$.

    Suppose for contradiction that $s_1 < \frac{1}{4}\sqrt{1-r_1}$. Because $s_1^2 + s_d^2 > \frac{1}{2}(1-r_1)$, and $s_d < 0$, we conclude that $s_d < -\frac{7}{16}\sqrt{\frac{(1-r_1)}{2}}$, but this would imply that 
    \[
        - (d-2)(1-r_1) \le s_1 + s_d < -\frac{3}{16}\sqrt{\frac{(1-r_1)}{2}},
    \]
    which is false for $1-r_1$ small enough.
    We conclude that $s_1 \ge \frac{1}{4}\sqrt{1-r_1}$, which we have seen implies the result.
\end{proof}
\begin{lemma}\label{lem:rhosurj}
    For $n = 4$, the map $\phi$ is surjective.
\end{lemma}
\begin{proof}
    When $n = 4$, we note that $A_3$ is a 2 dimensional simplex, whose boundary is the union of three line segments:
    \[
        L_1 = \{(\frac{1+t}{2},\frac{1-t}{2},0) : t \in [0,1]\}.
    \]
    \[
        L_2 = \{(1-2t,t,t) : t \in [0,\frac{1}{3}]\}.
    \]
    \[
        L_3 = \{(t,t,1-2t) : t \in [\frac{1}{3},\frac{1}{2}]\}.
    \]
    Each of these segments can be parameterized by a linear function from $[0,1]$ to $A_3$.
    Let $\gamma : [0,1] \rightarrow A_3$ be the closed piecewise linear curve which concatenates these parameterizations in order.

    Note that $\gamma$ has the property that for any $x \in A_3^o$ (the interior of $A_3$), $\gamma$ is not contractible in $A_3 \setminus \{x\}$.
    On the other hand, any curve in $A_4$ is contractible as $A_4$ is simply connected. 

    This implies that if we can find a \emph{lift} of $\gamma$ to $A_4$, i.e. a map $\hat{\gamma} : [0,1] \rightarrow A_4$ with the property that $\phi \circ \hat{\gamma}$ is homotopic to $\gamma$, then $\phi$ is surjective. Otherwise, we would have that there is some $x \in A_3$ so that $\phi$ is a well defined map from $A_4$ to $A_3 \setminus x$, so that $\hat{\gamma}$ would map to a noncontractible curve in $A_3\setminus x$, a contradiction.
    
    We now give an explicit lift $\hat{\gamma}$. Consider the line segments
    \[
        \hat{L}_1 = \{(1-2t,t,t,0) : t \in [0,\frac{1}{3}]\}.
    \]
    \[
        \hat{L}_2 = \{(\frac{1-t}{3},\frac{1-t}{3},\frac{1-t}{3},t) : t \in [0,\frac{1}{4}]\}.
    \]
    \[
        \hat{L}_3 = \{(t,\frac{1-t}{3},\frac{1-t}{3},\frac{1-t}{3}) : t \in [\frac{1}{4},1]\}.
    \]
    
    The endpoints of these line segments are $(1,0,0,0)$, $(\frac{1}{3},\frac{1}{3},\frac{1}{3},0)$, $(\frac{1}{4},\frac{1}{4},\frac{1}{4},\frac{1}{4})$. These points map to $(1,0,0)$, $(\frac{1}{2},\frac{1}{2},0)$, $(\frac{1}{3},\frac{1}{3},\frac{1}{3})$ respectively.
    
    Each of these line segments can be parameterized by a linear map, and if $\hat{\gamma}$ is the concatenation of these maps, then $\hat{\gamma}$ defines a simple closed curve.
    We now need to show that $\hat{\gamma}$ lifts $\gamma$.

    It can be seen using \Cref{lem:mult1} that $\phi(\hat{L}_i) \subseteq L_i$ for $i = 1, 2, 3$. For example, $\hat{L}_1$ maps to $L_1$ because for any $r \in L_1$, $r_4 = 0$. This implies that $p = \prod_{i=1}^4 (t-r_i)$ vanishes at 0, and so  $\delta_n(p)$ vanishes at 0 by \Cref{lem:mult1}, and therefore, $\phi(r) \in L_1$.
    Moreover, $\phi$ maps the two endpoints of $\hat{L}_i$ to the two endpoints of $L_i$ in order.
    This implies that when restricted to $\hat{L}_i$, $\phi$ is a homotopy equivalence between the natural parameterizations of $\hat{L}_i$ and $L_i$. 

    In conclusion, we are able to homotopy $\phi \circ \hat{\gamma}$ to $\gamma$ on each $L_i$ separately while preserving the endpoints of the $L_i$. We conclude by contractibility of $A_4$ that $\phi$ is surjective.
\end{proof}

\begin{thm}
    Let $T : \R[t]_{n,0} \rightarrow \R[t]_{4,0}$ be a diagonal zero-sum hyperbolicity preserver. Then there exists $\hat{T}:\R[t]_{n} \rightarrow \R[t]_{4}$ such that $\hat{T}$ is a diagonal hyperbolicity preserver, and $\hat{T}(f) = T(f)$ for all $f \in \R[t]_{n,0}$.

    Moreover, this is the case if and only if 
    \[
        g = T((t+n-1)(t-1)^{n-1})
    \]
    has real roots with three of the same sign.
\end{thm}
\begin{proof}
    It follows from \Cref{lem:necessary} that if $T$ is extendable, then  $g = T((t+n-1)(t-1)^{n-1})$
    has real roots of which 3 have the same sign.

    On the other hand, suppose that $g = T((t+3)(t-1)^{3})$
    has real roots of which $3$ have the same sign.

    If $g = at^4$ for some $a \in \R$, then for any $f \in \R[t]_{n,0}$, $T(f) = \ell(f)t^4$ for some linear function $\ell : \R[t]_{n,0} \rightarrow \R$, so we may take $\hat{T}(f) = \ell(f)t^4$ for any $f \in \R[t]_n$, and we have our desired extension.
Now suppose that $g$ has a nonzero root, so that     \[
        g = \prod_{i=1}^4 (t-r_i),
    \]
    where $r_1 > 0$ and $r_2, r_3 \ge 0$.
    By Lemma~\ref{lem:rhosurj}, we see that there exists $(s_1, s_2, s_3,s_4) \in A_4$ such that 
    \[
        \phi(s_1, s_2, s_3, s_4) = \frac{1}{r_1+r_2+r_3}(r_1, r_2, r_3), 
    \]
    which implies that for $\hat{g}(t) = \prod_{i=1}^4 (t-s_i)$ we have
    \[
        \delta_n(\hat{g}(t)) = g((r_1+r_2+r_3)t).
    \]
   Therefore,
    \[
        \delta_n(\hat{g}(\frac{1}{r_1+r_2+r_3}t)) = g(t).
    \]
    It follows from \Cref{lem:roots_extendable} that $T$ is extendable.
\end{proof}
\section{Spectrahedral Representations}
In this section, we will show the following theorem.
\spectrahedral*

We summarize some existing results \cite[Lemma 5.3]{kummer2021spectral} and \cite[Lemma 2]{branden2014hyperbolicity} in the following theorem:
\begin{thm}
    There is a matrix $B_{n,k}(x)$ of linear forms with the following properties:
    \begin{itemize}
        \item $B_{n,k}(x)m(x) = e_k(x) \delta$, where $m(x)$ is a vector whose entries are forms of degree $k-1$ (and whose first entry is $e_{k-1}(x)$), and $\delta$ is the coordinate vector whose first entry is 1.
        \item $B_{n,k}(x) \succeq 0$ iff $x \in \Lambda_{\vec{1}}e_k$.
        \item $\det(B_{n,k}(x)) = e_k(x) \prod_{S \subseteq [n] : |S| \le k} (\partial^S e_{k-1}(x))^{|S|!(n-|S|-1)}$.
    \end{itemize}
\end{thm}
We now prove the theorem:
\begin{proof}[Proof of \Cref{thm:spectrahedral}]
Fix $a \ge 0$, and let $D_{n,k}(x) = B_{n,k}(x) + a\ell(x)\delta \delta^{\intercal}$. We then have that
\[
    D_{n,k}(x) m(x) = (B_{n,k}(x) + \ell(x)\delta \delta^{\intercal})m(x) = (e_k(x) + \ell(x) e_{k-1}(x))\delta.
\]
We conclude that $e_k(x) + \ell(x) e_{k-1}(x)$ is a factor of $\det(D_{n,k}(x))$ (because whenever $e_k(x) + \ell(x) e_{k-1}(x)$ vanishes, $m(x)$ is in the kernel of $D_{n,k}(x)$, and $m(x)$ is generically nonzero). Moreover, because $\delta \delta^{\intercal}$ is rank 1, we have that $\det(D_{n,k}(x))$ is in fact degree at most 1 in the coefficients of $\ell$, and so we have that for some polynomial $q(x)$, $\det(D_{n,k}(x)) = (e_k(x) + \ell(x) e_{k-1}(x)) q(x),$
where $q(x)$ does not depend on $\ell$. Examining the case when $\ell = 0$, we conclude that $q(x) = \prod_{S \subseteq [n] : |S| \le k} (\partial^S e_{k-1}(x))^{|S|!(n-|S|-1)}$. 
Note that the fact that $q(x)$ is a product of directional derivatives of $e_{k-1}$ implies that it does not vanish on the interior of the hyperbolicity cone of $e_{k-1}$.
Also note that $e_{k-1}$ interlaces $e_k + \ell e_{k-1}$, and so the hyperbolicity cone of $e_k + \ell e_{k-1}$ is contained in that of $e_{k-1}$.
Therefore, $q(x)$ does not vanish on the interior of the hyperbolicity cone of $e_k + \ell e_{k-1}$. Together with the fact that $D_{n,k}(\vec{1}) = B_{n,k}(\vec{1}) + \ell(\vec{1}) \delta\delta^{\intercal}$ allows us to conclude that $D_{n,k}(x) \succeq 0$ iff $x$ is in the hyperbolicity cone of $e_k + ae_1e_{k-1}$, as desired.
\end{proof}
\subsection{Cubics}
\begin{thm}
Let $p$ be a symmetric hyperbolic cubic. $\Lambda_{\vec{1}}p$ is spectrahedral.
\end{thm}
\begin{proof}
    Consider the family of symmetric linear transformations defined by $T_{a}(x) = x - a e_1(x) \vec{1}$.
    Because $p$ is hyperbolic, there is some $t$ so that
    \[
        p(b - t\vec{1}) = 0, \text{ and }p(b - t'\vec{1}) \neq 0 \text{ for } t' > t.
    \]
    for any coordinate vector $b$.
    It suffices to show the result in the case when $T_t$ is invertible by a limiting argument.
    
    For this value of $t$, let $p' = p(T_t(x))$, which is also a symmetric hyperbolic cubic with $b \in \partial \Lambda_{\vec{1}} p'$.  It is clear that the hyperbolicity cone of $p'$ is linearly isomorphic to that of $p$, so in particular, if $\Lambda_{\vec{1}}p'$ is spectrahedral, so is $\Lambda_{\vec{1}} p$.
    
    We have that $p' = c_1 \tilde{e}_3 + c_2 \tilde{e}_1 \tilde{e}_2 + c_3 \tilde{e}_1^3$, for some $c_1, c_2, c_3 \in \R$ and $p'(b) = c_1 e_3(b) + c_2e_1(b)e_2(b) + c_3 e_1(b)^3 = c_3 = 0$.
    Hence, $p' = c_1 e_3 + c_2 e_1 e_2$.

    Because $b \in \Lambda_{\vec{1}}p'$, we also have that $D_{\vec{1}}p(b) = c_2$ has the same sign as $D_{\vec{1}}^2p(b) = c_1$. This implies that $p'$ has the same hyperbolicity cone as $c_1 e_3 + c_2 e_1 e_2$ for $c_1, c_2 \ge 0$, which we have established has a spectrahedral hyperbolicity cone in \Cref{thm:spectrahedral}.
\end{proof}

\section{Acknowledgements}
We would like to thank Mario Kummer for productive conversations.
\bibliographystyle{plain}
\bibliography{references.bib}
\end{document}